\documentclass[12pt]{article}
\usepackage[margin=1in]{geometry}
\usepackage{plain}
\usepackage{multirow}
\usepackage{graphicx}

\begin{document}
{
    \catcode`\@=11
    \gdef\curl{\mathop{\operator@font curl}\nolimits}
    \gdef\div{\mathop{\operator@font div}\nolimits}
    \gdef\Tr{\mathop{\operator@font Tr}\nolimits}
}
\newtheorem{theorem}{Theorem}
\newtheorem{definition}{Definition}
\newtheorem{lemma}{Lemma}
\newtheorem{result}{Result}
\newtheorem{proposition}{Proposition}
\newenvironment{proof}[1][Proof]{\noindent{\it #1.} }{\ \rule
{0.5em}{0.5em}\medskip\par}
\def\K{{\cal K}}
\def\Z{{\bf Z}}
\def\N{{\bf N}}
\def\C{{\bf C}}
\def\R{{\bf R}}
\def\E{{\bf E}}
\def\A{{\cal A}}
\def\U{{\cal U}}
\def\F{{\cal F}}
\def\J{{\cal J}}
\def\vb{\big(\!\big(}
\def\ve{\big)\!\big)}

\title{Model error in the LANS-alpha and
NS-alpha deconvolution models of turbulence}
\date{June 20, 2017}
\author{Eric Olson\thanks{
    Department of Mathematics and Statistics,
	University of Nevada,
	Reno, NV 89557, USA. {\it email:}{\tt\ ejolson@unr.edu}}}
\maketitle

\begin{abstract}
This paper reports on a computational study of the 
model error in the LANS-alpha and
NS-alpha deconvolution models of homogeneous 
isotro\-pic turbulence.
Computations are also performed for a new turbulence model 
obtained as a rescaled limit of the deconvolution model.
The technique used is to plug a 
solution obtained from direct numerical simulation of
the incompressible Navier--Stokes equations 
into the competing turbulence models and to then 
compute the time evolution of the resulting residual.
All computations have been done in two dimensions 
rather than three for convenience and efficiency.
When the effective averaging length scale 
in any of the models
is $\alpha_0=0.01$ the
time evolution of the root-mean-squared residual error grows 
as $\sqrt t$.
This growth rate similar to what would happen 
if the model error were given by a stochastic force.
When $\alpha_0=0.20$ the residual error grows linearly.
Linear growth suggests that the model error possesses
a systematic bias.
Finally, for $\alpha_0=0.04$
the residual error in LANS-alpha model exhibited linear growth;
however, for this value of~$\alpha_0$ the higher-order 
alpha models that were tested did not.
\end{abstract}

\section{Introduction}

Consider two dynamical systems
$$
	{du\over dt} = \F(u)
\qquad\hbox{and}\qquad
	{dv\over dt} = \widetilde \F(v)
$$
on a Hilbert space $V$ with norm $\|\cdot\|$.
Suppose the evolution of $u$ is given by 
exact dynamics
and the evolution of~$v$ according to some approximate 
dynamics.
Define the model error of the approximate dynamics as
the residual $R$ obtained by plugging the exact 
solution $u$ into the equation governing $v$.
Thus, 
\begin{equation}\label{resid}
	dR
	=du-\widetilde \F(u)dt
	=\big(\F(u)-\widetilde \F(u)\big)dt
\end{equation}
where by convention we take $R(0)=0$.
Specifically, consider the case where $\F$ is given by the 
two-dimensional incompressible Navier--Stokes equations 
and 
$\widetilde\F$ represents a particular alpha turbulence model.
The focus of this paper is whether, to what extent, and under what 
conditions do the residuals $R$ obtained through numeric 
computation behave qualitatively as spatially-correlated and 
temporally-white Gaussian processes.

This question is motivated, in part, by the analysis 
of Hoang, Law and Stuart~\cite{hoang2014} for the 
4DVAR data assimilation algorithm.
That analysis assumes $R=W$ where $W$ is a spatially-correlated 
and temporally-white Gaussian process and proceeds to show that
the inverse problem of finding the initial condition $u_0$ and 
the posterior distribution of $W$ is a continuous function 
of noisy observations of the velocity field.
In light of this result, we are interested whether the assumption 
$R=W$ is realistic when the residual error is given by actual 
turbulence models.
We are also motivated by the simple desire
to compare different turbulence models.
Stolz, Adams and Kleiser \cite{stolz2001} state that taking the
order for the NS-alpha deconvolution models to be $d=3$ already gives 
acceptable results while choosing the order larger than $5$ does not 
improve the results significantly.
We test this claim by examining the growth rate of $R$ for
different values of $d$ and comparing the 
results to a new rescaled limit of the deconvolution model which 
has an exponentially small consistency error.

The LANS-alpha model of turbulence is given by the equations
\begin{plain}\begin{equation}\label{lansalpha}\eqalign{
{\partial v\over \partial t}
	+ (\bar v\cdot\nabla)v+v_j \nabla \bar v_j=\nu\Delta v-\nabla p+f,\cr
\noalign{\medskip}
\nabla\cdot \bar v = 0
\qquad\hbox{where}\qquad
	v=(1-\alpha^2 \Delta) \bar v.\cr
}\end{equation}\end{plain}%
Here $v$ is the Eulerian velocity field, $\bar v$ is the
average velocity field, $\alpha$ is the averaging length scale,
$\nu$ is the kinematic viscosity, $p$ is the physical pressure
and $f$ is a body force.
Note that setting $\alpha=0$ yields the standard Navier--Stokes 
equations.
These equations,
originally called the viscous 
Camassa--Holm equations, were introduced
as a closure for the Reynolds averaged Navier--Stokes
equations by Chen, Foias, Holm, Olson, Titi and Wynne in 1998
through a series of papers~\cite{chen1998,chen1999a,chen1999b}.
At the same time, numerical simulations by Chen, Holm,
Margolin and Zhang~\cite{chen1999c} concluded that 
the LANS-alpha model also functions as an effective 
subgrid-scale model.  
Connections to the theory of global attractors and 
homogeneous isotropic turbulence and global attractors 
appear in \cite{foias2002}.  

Note that equations (\ref{lansalpha}) can be derived as
the Euler-Poincar\'e equations of an averaged Lagrangian
to which a viscous term, obtained by identifying the momentum
in the physical derivation, has been added.
This derivation further assumes that the turbulence is 
homogeneous and isotropic.
A body of theoretical and numerical literature on
the LANS-$\alpha$ model exists---see
\cite{cao2009,cheskidov2004,deugoue2012,hecht2008,hoang2014,
ilyin2003,lopes2015,marsden2003,mininni2005}
and references therein---that, among other things,
explores the dependency on $\alpha$ and the limit when 
$\alpha\to 0$,
relaxes the homogeneity and isotropy assumptions,
studies boundary conditions and boundary layers,
and treats other physical systems.
In summary, the LANS-alpha model
is a well-studied turbulence model that is suitable
for further study here.

To avoid a study of boundary layers we consider flows in 
domains with periodic boundary conditions.
To approximate homogeneous and isotropic turbulence we
choose a time-independent body forcing that has no regular 
patterns in space and for which the resulting flow 
undergoes complex time dependent behavior that in no way 
resembles the force, see Figure~\ref{forcefig}.
Since the body-force is time independent, it is natural
to suppose the statistics of the flow are stationary.  
While these assumptions are consistent with
the classical theories of fully developed turbulence
developed by Kolmogorov~\cite{kolmogorov1941} and
Kraichnan~\cite{kraichnan1967},
the possibility of intermittency
may lead to non-equilibrium and non-stationarity.
Moreover, while domains with periodic boundary conditions are obviously 
homogeneous, the presence of any non-zero forcing function has 
the potential to render the statistics of the resulting flow 
inhomogeneous.
As noted by Kurien, Aivalis and Sreenivasan~\cite{kurian2001}, 
see also Taylor, Kurien and Eyink~\cite{taylor2003},
even when the body forcing is zero,
turbulent flows in periodic domains can possess a certain 
degree of anisotropy.
It is hoped, therefore, that the stationarity, homogeneity and 
isotropy assumptions made in the derivation of the LANS-alpha and 
NS-alpha deconvolution models are well enough satisfied that 
the turbulence models studied here apply.
Viewed in a different way, our computations of the residual error 
may be seen as a test of these assumptions.

The NS-alpha deconvolution model of turbulence is
structurally the same as the LANS-alpha model, except
that the derivation allows for the more general 
filtering relationship between $v$ and $\bar v$ given by
\begin{equation}\label{NSalphafilt}
	\bar v=D_d (I-\alpha^2\Delta)^{-1} v
\end{equation}
where $D_d$ is the $d$-th order van Cittert approximate 
deconvolution operator
$$\displaystyle D_d=\sum_{n=0}^d \big(1-(1-\alpha^2\Delta)^{-1}\big)^n.$$
Note that setting $d=0$ yields the LANS-alpha model
and setting $\alpha=0$ again yields the incompressible
Navier--Stokes equations.
This model was introduced by Rebholz \cite{rebholz2007} 
as a helicity correction to higher-order Leray-alpha models.
It may also be seen as the $\alpha=\beta$ case of the 
alpha-beta models which have been the subject of recent
numerical work by Kim, Neda, Rebholz and 
Fried \cite{kim2011} and others.
For simplicity we don't consider the alpha-beta generalization here,
but instead focus solely on how the order $d$ of the
deconvolution operator affects the growth of the residual error.

In this paper we also study the limit turbulence model
obtained by identifying the effective averaging length scale 
$\alpha_0=\alpha/\sqrt{d+1}$ in the NS-alpha deconvolution model
and then taking $d\to\infty$ while holding $\alpha_0$ constant.
This results in a new turbulence model with an exponential 
smoothing filter given by
$$
	\bar v= \Big\{1-\exp\Big({\Delta^{-1}\over \alpha_0^2}\Big)\Big\} v
$$
with the same structure as LANS-alpha and NS-alpha deconvolution
models.  For convenience of terminology in the remainder of this 
paper, we will refer to this limit turbulence model as 
the {\it exponential-alpha model}.

Since the dynamics of the turbulence models 
considered here are 
deterministic, the model error represented by the residual $R$
is also deterministic.
To understand to what extent our computations support the
assumptions in \cite{hoang2014},
we now recall what happens when the model error is actually equal 
to a stochastic 
force and further what happens when it contains a systematic bias.
Suppose $\widetilde \F(u)dt=\F(u)dt-dW$
where $W$ is a $V$-valued $Q$-Brownian motion.
Here $V$ is an infinite dimensional Hilbert space and $Q$ is
a trace-class symmetric linear operator on $V$.
Let $\Xi$ be the underlying probability space.
For each $\xi\in\Xi$ we obtain a
residual realized by the sample path $R(t)=W(t;\xi)$.
Thus,
$$
	\E\big[ \|R(t)\|^2\big] 
		= \Tr \big[{\rm Cov}(R(t))\big]=t \Tr Q.
$$
If the model error also contains a systematic bias, then
\begin{equation}\label{H1}
R(t)=tF_b+W(t;\xi)
\end{equation}
where $F_b\in V$ is the bias.
In this case,
\begin{equation}\label{fiterr}
	\E\big[\|R(t)\|^2\big]
		=t^2 \|F_b\|^2+t \Tr Q.
\end{equation}
Note that the root-mean-squared residual error
${\cal E}_{\rm rms}(t)= \E\big[ \|R(t)\|^2\big]^{1/2}$
grows as $\sqrt t$ in the stochastic case and linearly
when there is systematic bias.

Returning now to the deterministic case,
let $\F$ represent
the dynamics of the two-dimensional incompressible
Navier--Stokes equations and $\widetilde\F$ be
the two-dimensional version of one of the alpha 
turbulence models described above.
As we are studying fully developed turbulence that arises from 
long-term evolution, it is reasonable to suppose $u_0$ 
lies on the global attractor $\A$ determined by the exact dynamics.
We further assume, for computational convenience, that all our 
solutions are $2\pi$-periodic with mean zero.

Thus, for each $u_0\in\A$ we obtain a solution $u(t,x)$
to the incompressible two-dimensional Navier--Stokes equations
of the form
$$
	u(t,x)=\sum_{k\in\Z^2\setminus\{0\}} u_k(t) e^{ik\cdot x}
\qquad\hbox{with}\qquad u_k(t)\in \C^2
$$
such that $u_k=\overline{u_{-k}}$ and 
	$k\cdot u_k=0$.
Foias and Temam show in \cite{foias1989}, 
that such solutions are analytic in time with values in 
a Gevrey class of functions.
This implies, see also \cite{foias2015,temam1995} and references therein,
that there exist 
constants $M_\alpha$ depending only on $\nu$ and $f$ such that
$$
	\|u\|_\alpha^2=\vb u,u\ve_\alpha
		\le M_\alpha
\qquad\hbox{where}\qquad
\vb u,v\ve_\alpha = 4\pi^2 \sum_{k\in\J} |k|^{2\alpha} u_k \overline{v_k}
$$
for all $t\in\R$ and $u_0\in\A$.
For notational convenience write $\|u\|=\|u\|_1$ and
$\vb u,v\ve=\vb u,v\ve_1$.
We take $V=\{\, u : \|u\|<\infty\,\}$ note that $R(t;u_0)\in V$
and also that the divergence-free condition implies this
norm is equivalent to the $H^1$ Sobolev norm.

The residual $R$ depends on the solution $u$ to the
exact dynamics, which in turn, depends on the initial condition $u_0$.
Since $u_0$ is unknown in the case of data assimilation, 
we may interpret the parameter $\xi$ in (\ref{H1}) as depending 
on an unknown $u_0\in\A$ distributed 
according some probability measure $\mu$.  
This leads to a natural definition of the root-mean-square 
residual error using ensemble averages as
\begin{equation}\label{ensemble}
	{\cal E}_{\rm rms}(t)=
	\big\langle \|R(t)\|^2\big\rangle^{1/2}
		= \Big\{\int_{\A} \|R(t;u_0)\|^2 d\mu(u_0)\Big\}^{1/2}.
\end{equation}
We are now able to state one of our main results:
{\it
If $\alpha_0$ is sufficiently small, then 
numerical computations show ${\cal E}_{\rm rms}(t)$ grows as 
$\sqrt t$; however, if $\alpha_0$ is too large, 
then ${\cal E}_{\rm rms}(t)$ grows linearly in time.
Moreover, even before taking ensemble averages our
computations show for each of the hundred different $u_0\in\A$ 
tested that $\|R(t;u_0)\|$ grows as $\sqrt t$ when $\alpha_0$ is
sufficiently small and linearly when $\alpha_0$ is large.}

Of course the computational fact that ${\cal E}_{\rm rms}(t)$ grows 
as $\sqrt t$ for small values of $\alpha_0$ does not imply the model 
errors in the corresponding alpha models are actually given by
Brownian motions.
Brownian motions have independent increments and are 
almost-surely nowhere differentiable.  
On the other hand the deterministic residuals studied here 
are differentiable and, as will briefly be shown, do not 
have independent increments.

Given $\delta>0$ define $Z_j =Z(\tau_j,\delta)$ where
\begin{equation}\label{taudef}
	\tau_j=j\delta
\qquad\hbox{and}\qquad Z(t,\delta)=R(t;u_0)-R(t-\delta;u_0).
\end{equation}
If the residual errors had independent increments 
then we would have
$
	\langle Z_j\otimes Z_{j+1} \rangle = 0
$
for any value of $\delta$.  The following 
argument shows this is not the case
provided $\delta$ is small enough.
Consider the function 
$r(t)=\vb R(t;u_0),\phi\ve$ where
$\phi\in V$.
In light of the analyticity of $u(t)$, the function
$r(t)$ is continuously differentiable with 
time derivatives that are uniformly bounded on the attractor.
Thus, for $l=0,1,2,\ldots$ we have
$$B_l=\sup\{\, |r^{(l)}(t)| : u_0\in\A,\ \|\phi\|=1\hbox{ and }t\in\R\,\}<\infty.
$$
Although $r'(t)$ may vanish for some choices of parameters, 
it is reasonable 
to suppose that the ensemble average of $r'(t)^2$ is positive.

In particular, we assume there exists 
	$\epsilon>0$ and $T>0$ such that
\begin{plain}\begin{equation}\label{muass}
	\sup\big\{\,\big\langle r'(t)^2\big\rangle^{1/2}: \|\phi\|=1\,\big\}
	\ge \epsilon
	\qquad\hbox{for all}\qquad t\in[0,T].
\end{equation}\end{plain}%
While (\ref{muass}) is satisfied  with $T=100\,000$
in each of our computations as seen in Table \ref{mineps}, 
a rigorous justification is outside 
the scope of the present work.
Intuitively, such a justification would depend on the support of 
$\mu$ including points which are not fixed points.  
If the support consists of points whose trajectories are 
chaotic, then plausibly (\ref{muass}) could hold for all $t\in\R$.
While we do not need this stronger condition, we note 
for the choice of parameters considered in our computations
that the long term evolution depicted in Figure \ref{forcefig} 
does indeed suggest that the attractor consists of chaotic trajectories.

By Taylor's theorem we have
\begin{plain}$$\eqalign{
	r(\tau_{j-1})&=r(\tau_j)-\delta r'(\tau_j)
		+{\textstyle{1\over 2}}\delta^2r''(c_1)
	\cr
	r(\tau_{j+1})&=r(\tau_j)+\delta r'(\tau_j)
		+{\textstyle{1\over 2}}\delta^2r''(c_2)
}$$\end{plain}%
where $\tau_{j-1}<c_1<\tau_j<c_2<\tau_{j+1}$.  It follows from
the mean value theorem that
\begin{plain}$$\eqalign{
	(Z_j\otimes Z_{j+1})(\phi,\phi)
		&=
		(r(\tau_j)-r(\tau_{j-1}))
		(r(\tau_{j+1})-r(\tau_{j}))\cr
		&=\delta^2 \big\{
			r'(\tau_j)^2+{\textstyle{1\over 2}} \delta r'(\tau_j)
				\big(r''(c_2)-r''(c_1)\big) 
			- {\textstyle{1\over 2}}\delta^2 r''(c_1)r''(c_2)\big\}\cr
		&=\delta^2 \big\{
			r'(\tau_j)^2+{\textstyle{1\over 2}} \delta(c_2-c_1)
				r'(\tau_j) r'''(c_3)
			- {\textstyle{1\over 2}}\delta^2 r''(c_1)r''(c_2)\big\}\cr
		&\ge \delta^2 \big\{
			r'(\tau_j)^2
			-\delta^2 M^2\big\}\cr
}$$\end{plain}%
where $M^2=B_1B_3+{\textstyle{1\over2}B_2^2}$.
Now, choosing $\delta\le \min\{ T, 2^{-1/2}\epsilon/M\}$ yields
\begin{plain}\begin{equation}\label{iscor}\eqalign{
	\sup\{\,\langle (Z_1\otimes Z_{2})(\phi,\phi)\rangle : \|\phi\|=1\,\}
	&\ge \delta^2\big\{\epsilon^2 
		- \delta^2 M^2\big\}
	\ge {\textstyle {1\over 2}}\delta^2\epsilon^2
	>0.
}\end{equation}\end{plain}%
Therefore, the increments are positively correlated provided $\delta$
is small enough.

Although the residual error in a deterministic turbulence model
does not consist of independent increments,
for the choice of flow-parameters considered here
the fact that 
solutions to the Navier--Stokes equations
have a 
sensitive dependence on initial conditions and seem to forget 
their initial conditions exponentially over time leads to 
the possibility that the 
increments may appear independent when $\delta$ is large.
In general, chaotic systems can produce time-series that 
are indistinguishable from Gaussian random noise when analyzed with 
any type of linear analysis---power spectrum, autocorrelation 
or probability distribution functions---see, for example, 
Sprott \cite{sprott2001}.
Therefore, instead of performing standard statistical tests 
often used to check for randomness in economic data such as the 
Ljung--Box $Q$-statistic \cite{box2015} or the variance-ratio 
test of Lo and MacKinlay \cite{lo1988},
we instead content ourselves with some simple descriptive 
statistics.

This paper is organized as follows:  Section~2 explores the
relation between $\alpha$ and $d$ in the NS-alpha deconvolution
model to show that the effective averaging length scale $\alpha_0$
depends on $d$ as $\alpha_0=\alpha/\sqrt{d+1}$.
We then derive the exponential alpha model by
taking the limit $d\to\infty$ while holding $\alpha_0$ fixed.
Section~3 describes the numerical methods used
to compute the solution $u$ of the two-dimensional incompressible
Navier--Stokes equations that will be plugged
into the turbulence models to compute the residual error.
In Lemma~\ref{aball} we show for small enough time steps 
that the resulting discrete dynamical system posses a 
global attractor.
Section~4 presents our computational results including
our main result on the growth rate of ${\cal E}_{\rm rms}(t)$.
Section~5 further describes the statistical properties 
of residual error.
The paper ends by summarizing our conclusions and stating
some plans for future work.

\section{The Effective Averaging Length Scale}

In this section we identify the effective averaging length scale 
in the NS-alpha deconvolution model 
as $\alpha_0=\alpha/\sqrt{d+1}$
and then use this 
identification to derive the exponential-alpha model as the
limit $d\to\infty$ holding $\alpha_0$ constant.
Any deconvolution model based 
on the smoothing filter $(1-\alpha^2\Delta)^{-1}$ will
have the same scaling between $\alpha$ and $\alpha_0$
and exponential-alpha model in the limit.
In particular, the effective averaging length scale $\alpha_0$
identified here applies equality well to 
three-dimensional fluid flows.

We consider the effects of the smoothing 
filter (\ref{NSalphafilt}) 
on regular $2\pi$-periodic functions with zero spatial
average in Fourier space.
Similar results could be obtained in more general settings, 
however, since our computations will be made for
$2\pi$-periodic domains, it is easiest to work in that 
setting from the beginning.
Henceforth, write the functions $v$ and $\bar v$ 
in terms of Fourier series as
$$
	v(x,t)=\sum_{k\in\Z^2\setminus\{0\}} v_k(t) e^{ik\cdot x}
\qquad\hbox{and}\qquad
	\bar v(x,t)=\sum_{k\in\Z^2\setminus\{0\}} \bar v_k(t) e^{ik\cdot x}.
$$
It follows the smoothing filter (\ref{NSalphafilt}) in 
the NS-alpha deconvolution model may be written as
$$
	\bar v_k = D_{d,k}(1-\alpha^2|k|^2)^{-1} v_k
$$
where
$$
	D_{d,k}
	= \sum_{n=0}^d \Big(1-{1\over 1+\alpha^2|k|^2}\Big)^n
	= \sum_{n=0}^d 
\Big({\alpha^2|k|^2\over 1+\alpha^2|k|^2}\Big)^n.
$$
Summing the above geometric series yields
$$
	{D_{d,k}\over 1+\alpha^2|k|^2} = 
	1-\Big({\alpha^2|k|^2\over 1+\alpha^2|k|^2}\Big)^{d+1}.
$$
Observe that
\begin{plain}$$
	{D_{d,k}\over 1+\alpha^2 |k|^2} \to 1 \qquad
	\matrix{
	\hbox{as}& d\to\infty
	&\hbox{holding $\alpha$ constant,}\cr
	\hbox{or as}& \alpha\to 0
	&\hbox{holding $d$ constant.}\cr}
$$\end{plain}%
At the same time, note that
$$
	{D_{d,k}\over 1+\alpha^2 |k|^2} \to 0
	\qquad\hbox{ as }\qquad \alpha\to\infty
	\quad\hbox{holding $d$ constant.}
$$
The above limits suggest there may be a tradeoff between
$\alpha$ and $d$ which could be used to identify an
effective averaging length scale $\alpha_0$ that depends on $d$.

\begin{figure}[h!]
	\centerline{\begin{minipage}[b]{0.75\textwidth}
    \caption{\label{dnkfig}%
	The smoothing filter in the NS-alpha
    deconvolution model in Fourier space for four different
    choices of $\alpha$ and $d$.}
	\end{minipage}}
	\centerline{\includegraphics[height=0.45\textwidth]{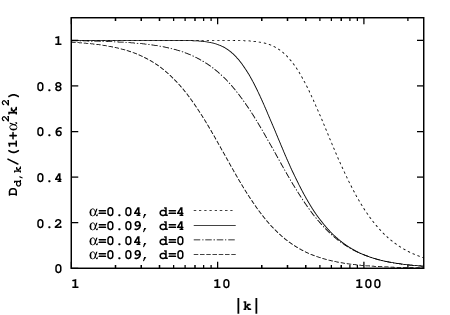}}
\end{figure}

Evidence of such a tradeoff is presented 
in Figure~\ref{dnkfig} where
four representative curves of the smoothing filter
$D_{d,k}/(1+\alpha^2|k|^2)$ are
plotted for $\alpha$ and $d$ where 
$\alpha\in\{0.04,0.09\}$ and $d\in \{ 0,4\}$.
Smaller values of $\alpha$ lead to a smoothing filter which is 
closer to 1 while larger values of $d$ also 
yield a filter which is closer to 1.
Moreover, the high-frequency attenuation is nearly the same 
when $\alpha=0.09$ and $d=4$ 
as it is when $\alpha=0.04$ and $d=0$.
This suggests a clear tradeoff between $d$ and $\alpha$
for which the microscales are the same.

We now perform an asymptotic analysis of $D_{d,k}/(1+\alpha^2|k|^2)$ 
as $k\to\infty$ to precisely identify 
the effective averaging length scale $\alpha_0$ that 
leaves the high-frequency attenuation of the smoothing filter 
unchanged as $d$ is varied.
Since
$$
	{D_{d,k}\over 1+\alpha^2|k|^2} = 
	1-\Big({\alpha^2|k|^2\over 1+\alpha^2|k|^2}\Big)^{d+1}
	\sim {d+1\over \alpha^2 |k|^2}
	\quad\hbox{as}\quad k\to\infty,
$$
one can rescale the filter by setting $\alpha=\alpha_0\sqrt{d+1}$
where $\alpha_0$ is constant to obtain an asymptotic decay 
that is independent of $d$ when $k\to\infty$.
The fact that $0.09\approx 0.04 \sqrt{4+1}$ now explains
why the high-frequency attenuation was nearly the same
for two of the curves in Figure~\ref{dnkfig}.

The above identification motivates the definition 
of the rescaled smoothing filter 
$H_{d,k}$ in terms of the effective averaging length scale 
$\alpha_0$ as
\begin{plain}$$\eqalign{
	H_{d,k}
&={D_{d,k}\over 1+(d+1)\alpha_0^2 |k|^2}
	=
		1-\Big(
			{(d+1)\alpha_0^2|k|^2\over
			1+(d+1)\alpha_0^2|k|^2}\Big)^{d+1}.
}$$\end{plain}%
The limit $d\to \infty$ holding $\alpha_0$ constant now 
leads to the non-trivial limit filter
\begin{plain}$$
	H_{\infty,k}=
	1-\lim_{d\to\infty}
		\Big(
			{(d+1)\alpha_0^2|k|^2\over
			1+(d+1)\alpha_0^2|k|^2}\Big)^{d+1}
	= 1 - \exp\Big({-1\over \alpha_0^2 |k|^2}\Big).
$$\end{plain}%
Moreover, for every $d\in {\bf N}\cup \{\infty\}$ we have that
$$
	H_{d,k} \sim {1\over \alpha_a^2 |k|^2}
	\qquad\hbox{as}\qquad k\to\infty
$$
and therefore the high-frequency
attenuation of this family of filters is independent of $d$.
Figure \ref{dnksfig} demonstrates how changing $d$ 
affects $H_{d,k}$ in the low modes without 
affecting the microscales.

\begin{figure}[h!]
	\centerline{\begin{minipage}[b]{0.75\textwidth}
    \caption{\label{dnksfig}%
	When $\alpha_0$ is fixed
	the cutoff for
	the high modes is independent of the order $d$.
	The vertical line denotes $|k|=1/\alpha_0$.
}
	\end{minipage}}
	\centerline{\includegraphics[height=0.45\textwidth]{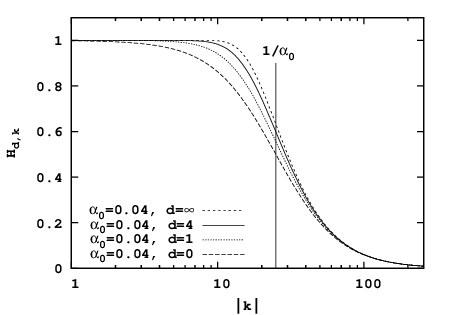}}
\end{figure}

We remark that this is the first time the exponential
filter $H_{\infty,k}$ has been derived and proposed for use in 
the context of turbulence modeling.
A Gaussian filter of the form $\exp(-\alpha_0^2|k|^2)$ was considered for the Leray and LANS-alpha models
by Geurts and Holm \cite{geurts2005}.
That filter is a mirror-image reflection of the one derived here:
Instead of preserving the cutoff in the high modes, 
it holds fixed the $d=0$ order of the low modes.

For smooth functions $u$ the filters $H_{d,k}$ have a consistency 
error given by
\begin{plain}$$
	u-H_{d,k} u =
	\cases{
		{\cal O}(\alpha_0^{2d+2})& for $d\in{\bf N}$\cr
		{\cal O}(\exp(-1/\lambda_0\alpha_0^2))&
		for $d=\infty$
}
	\qquad\hbox{as}\qquad
		\alpha_0\to 0
$$\end{plain}%
where~$\lambda_0$ is the smallest eigenvalue of the Stokes operator.
Therefore, the consistency error as a function of
the effective averaging length scale $\alpha_0$ is the same as
the results proved by Stolz, Adams and Kleiser 
\cite{stolz2001},
see also Dunca and Epshteyn \cite{dunca2006},
for the NS-alpha deconvolution model.
Moreover, the exponential filter obtained in the limit 
when $d\to\infty$ has exponentially small 
consistency error as $\alpha_0\to 0$.
We turn now to our computational results.

\section{Numerical Methods}

The vorticity formulations of the Navier--Stokes equations and 
the turbulence models described in the introduction are
particularly simple in two dimensions.
Using the notation
$\curl u=\partial u_2/\partial x_1 
	-\partial u_1/\partial x_2$, the 
two-dimensional incompressible Navier--Stokes equations
can be expressed as the scalar equation
\begin{equation}\label{vort}
	{\partial\omega\over\partial t}-\nu\Delta\omega
		+u\cdot\nabla\omega = g
	\qquad\hbox{where}\qquad
	\omega =\curl u
\end{equation}
and $g=\curl f$.
Similarly, the alpha turbulence models are all given by
\begin{equation}\label{mvort}
	{\partial m \over\partial t}-\nu\Delta m
		+\bar v\cdot\nabla m = g
	\qquad\hbox{where}\qquad
	H_d m = \curl \bar v
\end{equation}
and $H_d$ is the differential operator corresponding
to the symbol $H_{d,k}$ in the previous section.
We recall equations (\ref{mvort}) reduce to the LANS-alpha model 
when
$d=0$ and that the new exponential-alpha model is
obtained when $d=\infty$.

It is well known, see for example Temam~\cite{temam1995},
that the two-dimensional Navier--Stokes equations are
well-posed.  That is, 
these equations possess 
unique smooth solutions depending continuously on
the initial conditions
provided the force is sufficiently regular.
Foias, Holm and Titi \cite{foias2001} show that
three-dimensional LANS-alpha model 
is also well posed.
It follows, trivially, that the two-dimensional LANS-alpha
model is well posed.
Similar results hold for NS-alpha deconvolution model and 
also for the new exponential-alpha model.

\begin{figure}[h!]
	\centerline{\begin{minipage}[b]{0.75\textwidth}
    \caption{\label{forcefig}%
	Top left shows the contours
	of the force $g=\curl f$; in order following are 
	contours of $\omega$ at times $25\,000$,
	$50\,000$ and $100\,000$.}
	\end{minipage}}
    \medskip
\centerline{\hbox{\includegraphics[width=0.4\textwidth]{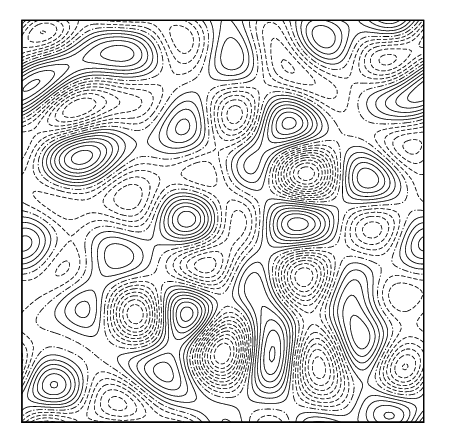}%
\includegraphics[width=0.4\textwidth]{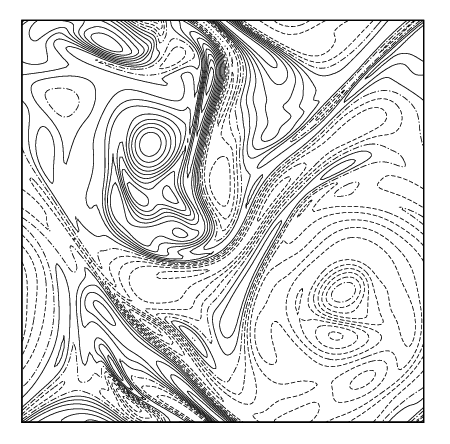}}}
\centerline{\hbox{\includegraphics[width=0.4\textwidth]{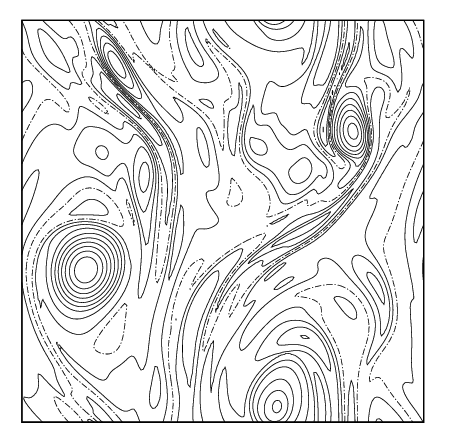}%
\includegraphics[width=0.4\textwidth]{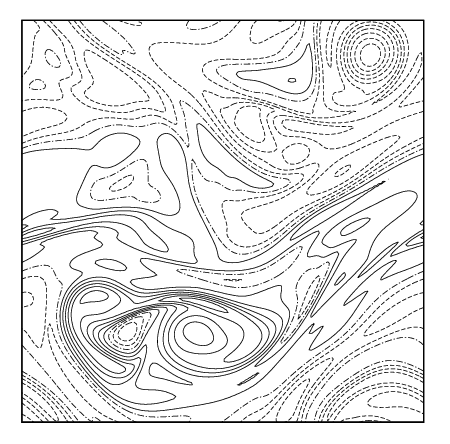}}}
\end{figure}

We shall consider a two-dimensional incompressible flow in a 
$2\pi$-periodic box forced by a 
$2\pi$-periodic body force.
Specifically, take $\Omega=[0,2\pi]^2$, the viscosity 
$\nu=0.0001$ and choose a
time-independent divergence-free body forcing $f$ supported on the
Fourier modes with $16 \le |k|^2\le 34$ 
such that $\|f\|_{L^2}=0.0025$ 
and for which the Grashof number is $G=\nu^{-2}\|f\|_{L^2}=250\,000$.
To obtain such a force, the amplitudes of the Fourier modes 
were chosen randomly and then rescaled to obtain the desired
Grashof number.
The exact function used here is depicted top left in 
Figure~\ref{forcefig} and originally appeared 
in Olson and Titi \cite{olson2003} where it was called $f_{25}$ to
indicate its support lied on an annulus about $|k|^2=25$ in 
Fourier space.

Since we have written our fluid equations 
in vorticity form, it is natural to compute the residual 
vorticity $\rho$ given by $\rho=\curl R$.
No generality is lost in doing this, 
because we may later recover $R$ by inverting the 
definition of $\rho$ subject to the condition $\nabla\cdot R=0$.
Note, as the velocity fields present in both the exact 
and approximate dynamics are divergence free,
so $R$ is divergence free.
Therefore, we plug $\omega$ into equation (\ref{mvort}) 
to obtain
$$
	d \rho=
	\big((\bar u\cdot\nabla\omega)-(u\cdot \nabla\omega)\big)dt.
$$
Integrating in time then yields
\begin{equation}\label{residR}
	\rho(t)=\int_0^t 
		\big((\bar u\cdot\nabla\omega)-(u\cdot \nabla\omega)\big)dt.
\end{equation}
It is worth remarking that at no point does the time evolution of 
the approximate dynamics~(\ref{mvort}) governing~$m$ 
enter into the computation of $\rho$.  
Indeed, sensitive dependence on initial conditions is well 
known for the dynamical systems studied here, which in turn, 
implies
there is no shadowing result that could be used to compare
separate evolutions of $\omega$ and $m$ over long periods of time.

Theoretically the evolution of $\omega$ should be determined by 
the exact dynamics of the two-dimensional incompressible 
Navier--Stokes equations.
As this is not possible in any numerical experiment, we therefore
consider two discrete dynamical systems
$$
	\omega^{n+1}=S(\omega^n)
\qquad\hbox{and}\qquad
	m^{n+1}=\tilde S(m^n)
$$
and suppose the discrete dynamics of 
$\omega^n$ are exact while $m^n$ evolves according to 
some approximate dynamics.  The discrete model error in the
approximate dynamics is then given by the residual
\begin{equation}\label{dexact}
	\rho^{n+1}=\rho^n+S(\omega^n)-\tilde S(\omega^n)
\qquad\hbox{where}\qquad
	\rho^0=0.
\end{equation}
Note that since (\ref{dexact}) represents an exact definition at
the discrete level, we obtain a method of computing the model error 
in the corresponding discrete alpha models to within the 
precision of the available floating-point arithmetic.

In the computations presented here the discrete
solution $\omega^n$, which we view as being governed by 
the exact dynamics, shall be 
given by the fully-dealiased spectral 
Galerkin method for approximating equations (\ref{vort}) in which 
the linear terms have been integrated exactly in time and the 
nonlinear term using an Euler method.
In particular, given $K\in\N$ fixed, let 
$$\K=\{\, (k_1,k_2)\in\Z^2\setminus\{0\} : 
	-K\le k_1,k_2\le K\,\}$$
and write
$$
\omega^n(x)=\sum_{k\in{\cal K}} \omega_k^n e^{i k\cdot x}
\qquad\hbox{and}\qquad
u^n(x)=\sum_{k\in{\cal K}} {i(k_2,-k_1)\over |k|^2} \omega_k^n
e^{i k\cdot x}.
$$
By definition, then, the exact dynamics are given by
\begin{equation}\label{exactd}
	\omega^{n+1}= S(\omega^n)
\qquad\hbox{for}\qquad
	n=0,1,2,\ldots
\end{equation}
where the discrete semigroup operator $S$
is given by 
\begin{plain}\begin{equation}\label{semid}\eqalign{
	S(\omega^n)_k=
		\Big\{\omega_k^n &- h 
		(u^n\cdot\nabla\omega^n)_k\Big\}e^{-\nu |k|^2 h}\cr
	&+ {2g_k\over \nu |k|^2}
		e^{-\nu |k|^2h/2}\sinh(\nu |k|^2h/2).
}\end{equation}\end{plain}%
While $S$ acts on the vorticity, an equivalent semigroup may be
defined which acts on the velocity.
For notational simplicity we 
shall refer to both semigroups as $S$.
Thus $u^{n+1}=S u^n$ shall mean $u^{n+1} = \curl^{-1} S \curl u^n$.

The corresponding approximate dynamics of our 
discrete alpha models are given by
\begin{equation}\label{approxd}
	m^{n+1}=\tilde S(m^n)
\qquad\hbox{for}\qquad n=0,1,2,\ldots.
\end{equation}
where
$$
	m^n=\sum_{k\in\K} m^n_k e^{ik\cdot x},\qquad
	\bar v^n = \sum_{k\in\K} 
	{i(k_2,-k_1)\over |k|^2} 
	H_{d,k} m_k e^{ik\cdot x}
$$
and
\begin{plain}$$\eqalign{
	\tilde S(m^n)_k=
		\Big\{m_k^n &- h 
		(\bar v^n\cdot\nabla m^n)_k\Big\}e^{-\nu |k|^2 h}\cr
	&+ {2g_k\over \nu |k|^2}
		e^{-\nu |k|^2h/2}\sinh(\nu |k|^2h/2).
}$$\end{plain}%

Note that $\omega_k^n$ may be viewed as a discrete approximation
of the continuous solution $\omega$ to equations (\ref{vort}) projected
onto the Fourier mode $\exp(ik\cdot x)$ at time $t_n=hn$ for some
time step $h>0$ and $m_k^n$ may be viewed as an identically discretized
approximation of the solution $m$ given by an alpha model.
Although more accurate and stable time stepping methods could be used,
the above is sufficient for our present study.

We now prove for $K$ fixed and $h$ small enough that
the discrete dynamical system (\ref{exactd}) possesses a unique 
global attractor $\A_{K,h}$.
This follows directly from following lemma which
shows the existence of an absorbing ball.
In order to keep track of the dimensional quantities which 
appear in the proof, define
$$\lambda_0=\min\{\, |k|^2:k\in\K\,\}=1$$ and 
recall that $g_k=0$ for $|k|^2>\lambda_M$
where $\lambda_M=34$.
\begin{lemma}\label{aball}
Let 
$$
	B=c_0 {\|f\|_{L^2}\over\nu\lambda_0^{1/2}}
\qquad\hbox{where}\qquad
c_0>6\lambda_M/\lambda_0.
$$
Given $K$ fixed and $L>0$, there is $h$ small enough and $N$
large enough such that $|w^0|<L$ implies $|w^n|<B$ for 
all $n\ge N$.
\end{lemma}
\begin{proof} From (\ref{exactd}) we have
$$
	\omega^{n+1}_ke^{\nu |k|^2 h}=\omega^n_k - h(u^n\cdot\nabla \omega^n)_k
		+{2g_k\over \nu |k|^2}\big(e^{\nu |k|^2 h}-1\big).
$$
Therefore
$$
	\sum_{k\in\K} 
	\omega^{n+1}_ke^{\nu |k|^2 h} e^{ik\cdot x}=\omega^n -
		h(u^n\cdot\nabla \omega^n)
		+\sum_{k\in\K}{2g_k\over \nu |k|^2}\big(e^{\nu |k|^2 h}-1\big)
		e^{ik\cdot x}.
$$
Let 
$$|\omega^n|^2=\|\omega^n\|_{L^2}^2=4\pi^2\sum_{k\in\K} |\omega^n_k|^2.$$
Since all norms are equivalent in finite dimensions
spaces, then there exists $C_K$ such that
$
	\big|u^n\cdot\nabla \omega^n\big|^2\le C_K |\omega^n|^4.
$
Moreover, since $\big(u^n\cdot\nabla \omega^n,\omega^n)=0$, then
$$
	\big|\omega^n-h(u^n\cdot\nabla \omega^n)\big|^2
		=\big|\omega^n\big|^2+h^2\big|u^n\cdot\nabla \omega^n\big|^2
		\le \big|\omega^n\big|^2 \big(1+C_K h^2|\omega^n|^2\big).
$$
For $h$ such that $\nu\lambda_M h<1$ we have $e^{\nu\lambda_Mh}<3$.
It follows that
\begin{plain}$$\eqalign{
		\Big|\sum_{k\in\K}{2g_k\over \nu |k|^2}\big(e^{\nu |k|^2 h}-1\big)
		e^{ik\cdot x}\Big|
	&=
	\Big(4\pi^2\sum_{k\in\K} {4|g_k|^2\over \nu^2|k|^4}
		\big(e^{\nu |k|^2 h}-1\big)^2\Big)^{1/2}\cr
	&\le
	\big(e^{\nu \lambda_M h}-1\big)
		{2|f|\over\nu\lambda_0^{1/2}}
	\le
	\nu\lambda_M h {6|f|\over\nu\lambda_0^{1/2}}.
}$$\end{plain}%
Therefore,
\begin{plain}$$\eqalign{
	|\omega^{n+1}|^2 e^{2\nu\lambda_0 h}
		&\le \big|\omega^n\big|^2\big(1 + C_Kh^2|\omega^n|^2\big)\cr
		&\quad+
			2\big|\omega^n\big|\big(1+C_Kh^2|\omega^n|^2\big)^{1/2}
			\nu\lambda_M h
		{6|f|\over \nu\lambda_0^{1/2}}\cr
		&\quad+
			\nu^2\lambda_M^2 h^2 {36|f|^2\over\nu^2\lambda_0}\cr
		&\le 
		\big|\omega^n\big|^2
		\big(1 + C_Kh^2|\omega^n|^2\big)
		(1 + \nu\lambda_0h) \cr
		&\quad+
		\Big({\lambda_M\over\lambda_0}+\nu\lambda_Mh\Big)
			\nu\lambda_Mh
			{36|f|^2\over\nu^2\lambda_0}.
}$$\end{plain}%
Now, if $B\le |\omega^n|\le L$, then
$$
	|\omega^{n+1}|^2 \le \alpha(h) e^{-2\nu\lambda_0h} |\omega^n|^2
$$
where
$$
	\alpha(h)=
		\big(1 + C_Kh^2 L^2\big)
        (1 + \nu\lambda_0h)
		+
        36\Big({\lambda_M\over\lambda_0}+\nu\lambda_Mh\Big)
            {\nu\lambda_Mh\over c_0^2}.
$$
Since $\alpha(h)\to 1$ as $h\to 0$ and
$$
	\alpha'(0)=\nu\lambda_0+36 {\nu\lambda_M^2\over c_0^2\lambda_0}
		<2\nu\lambda_0,
$$
there is $h$ small enough that 
$\gamma=\alpha(h)e^{-2\nu\lambda_0h}<1$ as well as $\nu\lambda_Mh<1$.
Let $N$ be large enough that $L\gamma^N<B$.
Since $|\omega^n|<B$ implies
$
	|\omega^{n+1}|^2\le \gamma B^2 < B^2,
$
once $|\omega^n|$ falls below $B$ it stays below $B$.
It follows that $|\omega^{n+1}|<B$ for all $n\ge N$,
which completes the proof of the lemma.
\end{proof}

Up to the constant $c_0$, the bound on $B$ 
given above is the same as the usual estimate
on the absorbing ball of the two-dimensional incompressible
Navier--Stokes equations (\ref{vort}), see, for example \cite{temam1995}.
Note that the estimate
on the size of $h$ depends on $C_K$, which we have not explicitly
computed here.  Lemma~\ref{aball} is important theoretically; however,
as with other {\it a~priori\/} estimates of this type, 
the resulting bounds on $h$ and $B$ differ by many orders of 
magnitude from those suggested by the numerics.
A similar theorem could be proved about the discrete
alpha models.
As only %the exact solution 
$\omega^n$ is used when computing 
the residual $\rho^n$, we omit that theorem and proof.

Consider next the computation of the discrete residual.  
For the discrete dynamical systems given by (\ref{exactd}) and
(\ref{approxd}) we obtain
\begin{equation}\label{residd}
	\rho^{n+1}_k
	=\rho^n_k+h \Big\{(\bar u^n\cdot\nabla w^n)_k
	- ( u^n\cdot\nabla w^n)_k\Big\}
	e^{-\nu |k|^2 h}.
\end{equation}
We remark again that, aside from the model error 
which we are trying to compute,
the only error which enters into the computation of
$\rho^n$ comes from the rounding present in 
the floating-point arithmetic.
As in the continuous case, the discrete residual velocity
may be obtained from the residual vorticity.  In particular,
$R_k^n = i\rho_k^n(k_2,-k_1)/ |k|^2$.

Note that $R_k^n$ reflects the modeling error made when replacing
our fully discrete dynamical system by a similarly discretized
alpha model.
Similar results should hold if other numerical schemes, 
such as BDF2, were used.
The method described here could also be used to compute 
the evolution of the residual error in other turbulence models.
Moreover, it should be possible to track the residual error present in 
a particular discretization of a continuous dynamical system,
for example, by comparing different numerical schemes.
For work along these lines please see Banks, Hittinger, Connors
and Woodward \cite{banks2013} and references therein.
While it may be possible that similar techniques could be used
to relate the discrete 
residuals $R_k^n$ that we compute here to the model error
in the fully-continuous alpha models, 
we do not pursue this direction of inquiry.

The computations which appear in this paper were implemented
using the MIT/Intel Cilkplus parallel processing extensions to 
the C programming language and compiled using GCC version 5.1.
The fast Fourier transforms used to compute the non-linear 
term were performed using the FFTW3 software library.
All computations were carried out using IEEE~754 double-precision 
floating point on the PDE Wulf cluster and the UNR Grid 
at the University of Nevada Reno.
The final computations presented here took a total of $38\,400$
core-hours of processing time using Intel Xeon E5-2650 CPUs.
The Navier--Stokes solver described in \cite{olson2003}
was used to verify the correctness of our computations.

The specific discretization considered here uses
a~$256\times 256$ spatial grid with~$K=85$ and a 
time step of~$h=25/4096$.
For practical reasons $h$ has been taken to be 
many orders of magnitude larger than the bounds 
given in Lemma~\ref{aball}.
Numerically, this choice of parameters leads
to a stable numerical scheme with a 
Courant--Friedrichs--Lewy condition number of
$$
	{\rm CFL}={Kh\over 2\pi}\sup\big\{|u^n(x_{ij})|_1: 
		x_{ij}\in\Omega\hbox{ and }
		t_n\le 100\,000 \big\} \approx 0.18
$$
where 
$x_{ij}=2\pi (i,j)/256$
and $|(u_1,u_2)|_1=|u_1|+|u_2|$.  We henceforth assume
that (\ref{exactd}) possesses a 
global attractor $\A_{K,h}$ suitable for our study.
Starting from the initial condition $u^0(x)=0$, we
obtain by time $t=25\,000$ a complex time-dependent velocity field
whose 
statistical properties appear to have reached a steady state.  

To further characterize the time scales in our computation, we 
estimate the eddy turnover time $\tau$ using the
definition of Gesho, Olson and Titi~\cite{gesho2015} as
$$\tau=4\pi^2\sum_{r=1}^\infty r^{-1}E(r)/
\Big(\sum_{r=1}^\infty E(r)\Big)^{3/2}
	\approx 92.05,
$$
where $E(r)$, see Figure \ref{wspecfig}, is the 
time-averaged energy spectrum given by
$$
	E(r)={4\pi^2\over {T-T_0}}\int_{T_0}^T
		\sum_{k\in\J_r}
		|u_k(t)|^2 dt
$$
averaged from $T_0=25\,000$ to $T=100\,000$
where 
$$\J_r= \{\,k\in\Z^2 :r-0.5<|k|\le r+0.5\,\}.$$
Note 
that the flow undergoes an additional $814$ eddy turnovers
on this time interval.
We presume, therefore, that
$u^n$ lies very near the global attractor of our
discrete dynamical system for $t_n\ge 100\,000$.
We now describe the discrete ensemble averages that will be 
used to compute the root-mean-squared residual error in 
our discrete alpha models.

\begin{figure}[h!]
	\centerline{\begin{minipage}[b]{0.80\textwidth}
    \caption{\label{pispin}%
	Long-time evolution of 100 solutions of the two-dimensional
	Navier--Stokes equations with randomly chosen initial
	conditions leading to 100 different points on the attractor.
	Note that the graph has been
    broken and data omitted between times $t_n=15\,000$ 
	and $t_n=85\,000$.}
	\end{minipage}}
	\centerline{\includegraphics[height=0.45\textwidth]{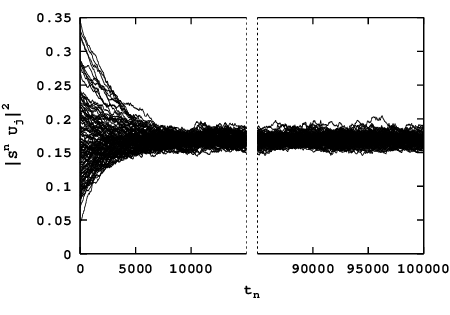}}
\end{figure}

The set $\U$, which forms the support of the probability 
measure $\mu$ used to define our ensemble averages,
was taken to consist of 100 points, each obtained by 
choosing a random velocity field $U_j$ and then evolving that
field forward $T=100\,000$ units in time.
Figure \ref{pispin} shows the
evolution of $\|S^n U_j\|_{L^2}^2$ for $j=1,2,\ldots,100$.
The statistical properties of the energy appear to have 
reached a steady state by $t_n=25\,000$ and by time $t_n=100\,000$ 
each flow has undergone approximately $1000$ large-eddy turnovers.  
We presume, for the same reasons as before, that each element of
$$
	\U=\{\, S^{n} U_j : j=1,2,\ldots,100\hbox{ and }t_n=100\,000\,\}
$$
is near the discrete global attractor $\A_{K,h}$.

\begin{figure}[h!]
	\centerline{\begin{minipage}[b]{0.75\textwidth}
    \caption{\label{ppara}%
	Locations in the energy-enstrophy plane of the 100
	points on the attractor $u_0\in\U\subseteq\A$ used 
	for the ensemble averages.}
	\end{minipage}}
	\centerline{\includegraphics[height=0.45\textwidth]{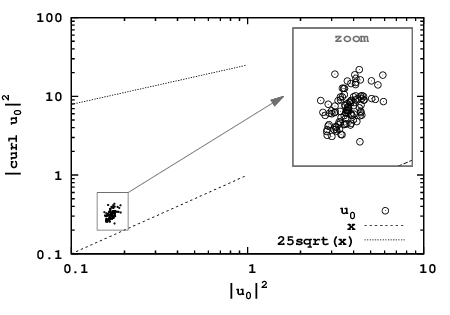}}
\end{figure}

Before proceeding we further characterize 
the ensemble averages used in our computations.
Figure \ref{ppara} plots the points $u_0\in\U$
in the energy-enstrophy plane.
The fact that all 100 points lie between the parabola and the
line is consistent with the analysis of Dascaliuc,
Foias and Jolly \cite{dascaliuc2005} on the location of
the global attractor for the two-dimensional
incompressible Navier--Stokes equations.
Observe that the points are clustered together in
a small region of the plane, but appear randomly 
distributed within that region.  

The discrete ensemble averages  
may be defined as follows.
For each $u_0\in\U$ let $R_k^n(u_0)$ be the
residual obtained by plugging the solution $u_k^n$
with $u^0=u_0$ into (\ref{residd}).
Take $\mu$ in (\ref{ensemble}) 
to be the uniform probability measure supported on $\U$.
It follows that
\begin{equation}\label{ensembled}
	{\cal E}_{\rm rms}^n=
		\big\langle \|R^n\|^2 \big\rangle^{1/2}
		= 
		\Big\{{4\pi^2\over |\U|}
		\sum_{u_0\in\U} 
			\sum_{k\in\K} |k|^2|R_k^n(u_0)|^2\Big\}^{1/2}
\end{equation}
where $|\U|=100$ denotes the cardinality of $\U$.  
Characterizing how ${\cal E}_{\rm rms}^n$ 
depends on $t_n$ will be the main focus of the
computational results in the next section.

\section{Computational Results}\label{compres}

For our numerical study, we compute the root-mean-squared residual 
error for nine different turbulence models, determined
by taking $d$ and $\alpha_0$ such that $d\in\{0,4,\infty\}$ 
and $\alpha_0\in\{0.01,0.04,0.20\}$.
Intuitively, 
for small values of $\alpha_0$ 
we expect an alpha model to function as a 
subgrid-scale model \cite{chen1999c,kim2011} and 
for large values of $\alpha_0$
as a Reynolds stress closure 
\cite{chen1998,chen1999a,chen1999b}.
To see how the different values of $\alpha_0$ considered in our
numerical experiments compare with the energetics of a
typical flow on the global attractor of~(\ref{exactd}), Figure~\ref{wspecfig} 
plots the energy spectrum $E(r)$ of the solution $u^n$ with initial 
condition $u^0=0$ averaged between times $T_0=25\,000$ and 
$T=100\,000$ against the vertical lines $|k|=1/\alpha_0$.
Each of these vertical lines represent the wavenumber
at which the Fourier modes are attenuated by 50 percent
in the smoothing filter of the original LANS-alpha model.
As previously illustrated in Figure~\ref{dnksfig}, larger values 
of $d$ lead to slightly less attenuation at this wavenumber.
We remark
that the smallest averaging
length scale $\alpha_0=0.01$ leads to
smoothing filters affecting modes in the dissipation
range of the energy spectrum, that $\alpha_0=0.04$ 
also affects modes
in the inertial range and that the relatively large value of 
$\alpha_0=0.2$ affects all the modes including those in the forcing 
range and inverse cascade.  

\begin{figure}[h!]
	\centerline{\begin{minipage}[b]{0.75\textwidth}
    \caption{\label{wspecfig}%
	The average energy spectrum of $u$ in relation to
	the wavenumbers $|k|=1/\alpha_0$ corresponding to
    three choices of $\alpha_0$.}
	\end{minipage}}
	\centerline{\includegraphics[height=0.45\textwidth]{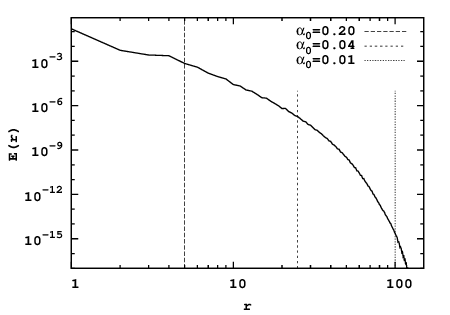}}
\end{figure}

Before computing the ensemble averages defined by equation
(\ref{ensembled}) it is informative to directly examine the 
residual error computed along a single representative
trajectory lying on the global attractor.
Let $u_0\in\A_{K,h}$ be fixed.  
Curves showing the evolution of the norm 
of $R^n(u_0)$ for 
the nine different turbulence models studied
are plotted in Figure~\ref{pltpath}.
For $\alpha_0=0.20$ the residual error is the largest and appears to
grow linearly with time after $t_n\ge 4000$ for each value of $d$.
For $\alpha_0=0.04$ the curves group together in the middle
of the graph and appear to grow as $\sqrt{t_n}$.  In this group
the curve corresponding to $d=0$ deviates from the 
other two when $t_n\ge 20000$ and starts to grow at a slightly
faster rate.  This deviation, though slight in the log-log plot,
is significant as further numerics shall indicate.
For $\alpha_0=0.01$ the residual error is
the least and separate curves appear at bottom of the graph.
Each of these curves
appear to grow as $\sqrt{t_n}$ over the entire range.
Therefore, even without taking ensemble averages, the 
differences in the growth rates of the residual error
described in our main result can be observed for the 
different turbulence models.

\begin{figure}[h!]
	\centerline{\begin{minipage}[b]{0.80\textwidth}
    \caption{\label{pltpath}%
	Evolution of the residual error along a single trajectory.
	The top three curves correspond to $\alpha_0=0.20$, the middle
	three to $\alpha_0=0.04$ and the bottom three to 
	$\alpha_0=0.01$.  Values for $d$ are as indicated.}
	\end{minipage}}
	\centerline{\includegraphics[height=0.45\textwidth]{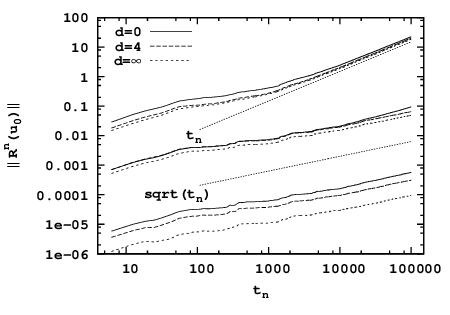}}
\end{figure}

\begin{table}[h!]
	\centerline{\begin{minipage}[b]{0.75\textwidth}
    \caption{\label{mineps}%
	Numerical lower bounds on $
	\sup\big\{\,\langle r'(t)^2\rangle^{1/2}:\|\phi\|=1\,\big\}$ 
	for $t\in[0,T]$ where $T=100\,000$ 
	obtained by taking $\phi\in\{\phi_1,\phi_2,\phi_3\}$.
	}
	\end{minipage}}
	\medskip
    \centering
    \begin{tabular}{ l c | c c c  }
    \hline\hline
\vphantom{\raise8pt\hbox{H}}%
	$\phantom{.}\alpha_0$&$d$&$\phi_1$&$\phi_2$&$\phi_3$\\ [0.5ex]
	\hline
\vphantom{\raise8pt\hbox{H}}%
0.01&0&$3.7383\times10^{-09}$&$3.7421\times10^{-09}$&$2.3073\times10^{-09}$ \\
0.01&4&$9.3591\times10^{-10}$&$7.0558\times10^{-10}$&$5.1470\times10^{-10}$ \\
0.01&$\infty$&$1.3167\times10^{-10}$&$1.1054\times10^{-10}$&$7.1270\times10^{-11}$ \\
0.04&0&$6.9357\times10^{-07}$&$8.1933\times10^{-07}$&$4.0261\times10^{-07}$ \\
0.04&4&$8.6752\times10^{-07}$&$5.9190\times10^{-07}$&$4.5601\times10^{-07}$ \\
0.04&$\infty$&$5.5177\times10^{-07}$&$3.4297\times10^{-07}$&$2.6645\times10^{-07}$ \\
0.2&0&$6.7480\times10^{-05}$&$2.0257\times10^{-04}$&$4.8341\times10^{-05}$ \\
0.2&4&$4.5501\times10^{-05}$&$1.8437\times10^{-04}$&$3.5777\times10^{-05}$ \\
0.2&$\infty$&$4.5846\times10^{-05}$&$1.6867\times10^{-04}$&$3.5809\times10^{-05}$ \\
[1ex]
\hline
\end{tabular}
\end{table}

Before proceeding to the computation of ${\cal E}_{\rm rms}^n$,
we verify that our ensemble averages satisfy the 
assumption (\ref{muass}) that
$\langle r'(t)^2\rangle^{1/2}$ has a positive lower bound
which is uniform in time.
As it is impossible to take the supremum over all values of 
$\phi$ choose
$\phi\in \{\phi_1,\phi_2,\phi_3\}$ with
\begin{equation}\label{zetav}
\phi_1=
{P_H(u_0\cdot\nabla u_0)\over
\|P_H(u_0\cdot\nabla u_0)\|}
	,\qquad
\phi_2={f\over \|f\|}
	\qquad\hbox{and}\qquad
\phi_3={u_0\over\|u_0\|}.
\end{equation}
Here $u_0={\rm curl}^{-1}\omega(T)$ where 
$\omega$ is the solution depicted in Figure \ref{forcefig}
at $T=100\,000$
and $P_H$ is the $L^2$ projection onto the 
divergence-free elements of $V$.
The computational results given in Table \ref{mineps} indicate 
that all three choices of $\phi$ yield similar 
minimum values for $\langle r'(t)^2\rangle^{1/2}$ and
that each of these are positive.
Moreover, with one minor exception, the lower bounds on 
$\langle r'(t)^2\rangle^{1/2}$ 
decrease as both $\alpha_0$ decreases and as $d$ increases.

\begin{figure}[h!]
	\centerline{\begin{minipage}[b]{0.80\textwidth}
    \caption{\label{pravg}%
	Evolution of ${\cal E}_{\rm rms}^n$ for nine different
	choices of parameters.  The top three 
	curves correspond to $\alpha_0=0.20$, the middle
	three to $\alpha_0=0.04$ and the bottom three to 
	$\alpha_0=0.01$.  Values for $d$ are as indicated.}
	\end{minipage}}
	\centerline{\includegraphics[height=0.45\textwidth]{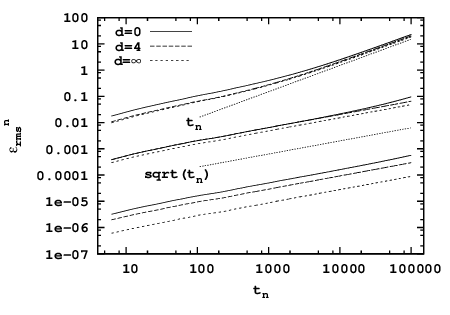}}
\end{figure}

We turn now to our main result, the computation of the 
root-mean-square residual error.
Curves showing the evolution of ${\cal E}_{\rm rms}^n$ for
$d\in\{0,4,\infty\}$ and $\alpha_0\in\{0.01,0.04,0.20\}$
are plotted in Figure \ref{pravg}.  
In this figure the curves appear smoother but are otherwise
similar to those in Figure \ref{pltpath}, especially for
large values of $t_n$.
To further characterize the growth of the residual error,
we find a least squares fit for the parameters $C_1$ and $C_2$ 
such that 
$$
	({\cal E}_{\rm rms}^n)^2\approx C_2 t_n^2 + C_1 t_n.
$$
If the residual error were comprised of a stochastic 
force plus a systematic bias,
then comparing with (\ref{fiterr}) would allow us to 
estimate ${\Tr Q}\approx C_1$ and ${\|F_b\|^2}\approx C_2.$
Since the residual actually comes from a deterministic dynamical 
system, this is not the case.  
However, the intuitive notion that a good turbulence model
should have an unbiased residual error implies that the
term represented by $C_2 t_n^2$ should be small compared to $C_1 t_n$.
To characterize the relative size of these two terms at the end 
of each computational run we define the dimensionless ratio 
$\eta=C_2T/C_1$.

Table \ref{leastsq} reports the values of $C_1$, $C_2$
and $\eta$ for the computational runs given in Figure \ref{pravg}.
When $\alpha_0$ is fixed, the estimates of
$C_1$ and $C_2$ have similar orders of magnitude for
different values of $d$ that generally decrease as $d$ increases.
The values of $\eta$ are nearly the same when $d=4$
or $d=\infty$ but differ when $d=0$.  
It should be pointed out that,
although $\eta$ depends directly on the length 
$T=100\,000$ of the computational run, it is still 
meaningful to compare the relative sizes of $\eta$ for 
different choices of $\alpha_0$ and $d$ while keeping 
$T$ fixed.

\begin{table}[h!]
	\centerline{\begin{minipage}[b]{0.75\textwidth}
    \caption{\label{leastsq}%
	Least squares fit of $({\cal E}_{\rm rms}^n)^2\approx
		C_2 t_n^2 + C_1 t_n$ and the ratio
		$\eta=C_2T/C_1$ at the end
		of the computation when $T=100\,000$.
	}
	\end{minipage}}
	\medskip
    \centering
    \begin{tabular}{ l c | c c | r }
    \hline\hline
\vphantom{\raise8pt\hbox{H}}%
    $\phantom{.}\alpha_0$ & $d$ & $C_1$ & $C_2$ 
		& $\eta\phantom{2.}$ \\ [0.5ex]
    \hline
\vphantom{\raise8pt\hbox{H}}%
$0.01$&$0$&
$2.513\times10^{-12}$&$7.086\times10^{-18}$&0.28 \\ 
$0.01$&$4$&
$8.489\times10^{-13}$&$2.733\times10^{-19}$&0.03 \\ 
$0.01$&$\infty$&
$7.723\times10^{-14}$&$1.698\times10^{-20}$&0.02 \\ 
$0.04$&$0$&
$4.066\times10^{-08}$&$5.010\times10^{-13}$&1.23 \\ 
$0.04$&$4$&
$4.108\times10^{-08}$&$1.100\times10^{-14}$&0.03 \\ 
$0.04$&$\infty$&
$2.306\times10^{-08}$&$2.747\times10^{-15}$&0.01 \\ 
$0.2$&$0$&
$1.179\times10^{-04}$&$5.019\times10^{-08}$&42.56 \\ 
$0.2$&$4$&
$4.175\times10^{-05}$&$3.985\times10^{-08}$&95.47 \\ 
$0.2$&$\infty$&
$4.319\times10^{-05}$&$3.358\times10^{-08}$&77.75 \\ 
[1ex]
\hline
\end{tabular}
\label{regress}
\end{table}

The similarity between the entries in Table~\ref{leastsq}
for $d=4$ and $d=\infty$ is consistent with~\cite{stolz2001} 
wherein it is reported that
for all tested applications $d=3$ already gives acceptable 
results, and that choosing $d$ larger than $5$ does not 
improve the results significantly.
The value $\alpha_0=0.20$ leads to $\eta\gg 1$ for every choice
of $d$, which suggests some sort of linearly growing bias 
dominates the residual error when $\alpha_0$ is large.
When $\alpha_0=0.04$ and $d=0$ the time evolution of the residual
error deviated slightly from the line $\sqrt{t_n}$ in 
Figure \ref{pravg}.
For this choice of parameters Table \ref{leastsq} indicates
that $\eta\approx 1.23$. 
This means that by the end of the computational run
the linearly growing part of the residual error contributes 
more than 50 percent to the total error.  While the exact
balance between the two terms depends on $T$,
it is interesting that the value of $\eta$ in
the $d=0$ case differs from the $d=4$ 
and $d=\infty$ cases by a couple orders of magnitude
when $\alpha_0=0.04$.
While the value $\alpha_0=0.01$ leads to
$\eta<1$ for every choice of $d$, it is again
notable that $d=0$ leads to a larger $\eta$.
We conclude by observing that $\eta\ll 1$ implies the linearly
growing term intuitively identified as bias is 
negligible over a time period of more than $1000$ large-eddy turnovers 
and the residual error, though deterministic, behaves as 
if the model error were given by stochastic force.
The next section provides additional analysis
which clarifies this point and further characterizes the bias as
well as the time and spatial correlations in the residual error.

\section{Further Analysis}\label{furan}

This section 
further
characterize the residual errors computed in the previous section.
By expanding upon the previous section, our goal is to determine to what 
extent the deterministic model errors considered here behave
as the spatially-correlated and temporally-white Gaussian processes
employed in theoretical works on data assimilation such as 
\cite{hoang2014}.

\begin{figure}[h!]
	\centerline{\begin{minipage}[b]{0.75\textwidth}
    \caption{\label{bias}%
	Contours of $\langle\rho^n\rangle$ at time $t_n=100\,000$ 
	for selected set of values for $\alpha_0$ and $d$.}
	\end{minipage}}
    \medskip
\centerline{\hbox{\includegraphics[width=0.4\textwidth]{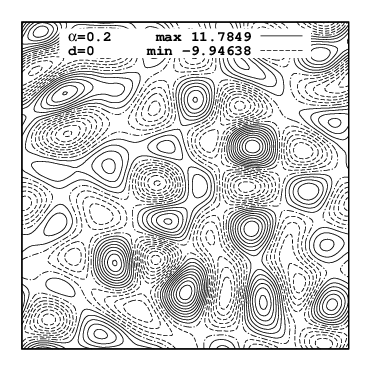}%
\includegraphics[width=0.4\textwidth]{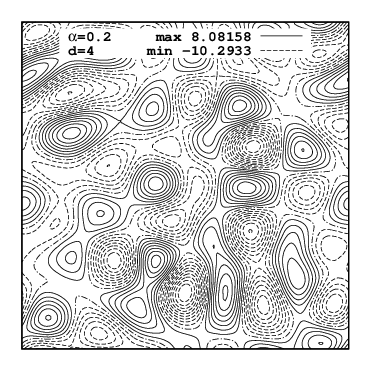}}}
\centerline{\hbox{\includegraphics[width=0.4\textwidth]{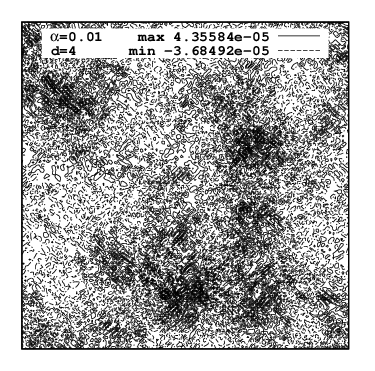}%
\includegraphics[width=0.4\textwidth]{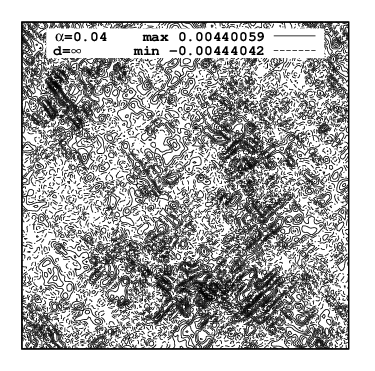}}}
\end{figure}

Let us begin by examining the ensemble averages of the residual error 
$$
	\langle \rho^n\rangle = 
		{1\over |\U|}\sum_{u_0\in\U} \rho^n(u_0)
$$
at the end of each computational run when $t_n=100\,000$.  
Figure \ref{bias} 
illustrates 
four examples from the nine choices of parameters studied:
two cases when $\eta\ll 1$ and when $\eta\gg 1$. 
When $\eta\ll 1$ the ensemble average of the residual error
looks almost random.
When $\eta\gg 1$, the ensemble average of the residual error
looks remarkably like the forcing function
depicted in Figure \ref{forcefig}.
While it is not surprising that the residual error
has a bias around the body forcing when the averaging
length scale is large, it is surprising that this bias
is of opposite signs when $d=0$ compared to when $d=4$.
This change in sign suggests an alpha model may
exist for some fractional value of $d\in (0,4)$ that 
minimizes the bias around the body forcing for large values 
of $\alpha_0$.
Since taking $\alpha_0$ large translates into computational
savings, increasing the accuracy in this situation by removing 
bias could be of practical importance.
Note that 
without taking ensemble averages the plots of
$\rho^{n}$ when $t_n=100\,000$ look essentially the same as 
those depicted for the ensemble averages in Figure \ref{bias}.
In particular,
the sign differences between $d=0$ and $d=4$ when 
$\alpha_0=0.2$ are consistent 
across all 100 independent trajectories considered in 
our study.
Finally it is worth mentioning that the sign of the bias 
around the forcing function doesn't oscillate in time,
but instead becomes more and distinct.

Now describe the degree to which the residual
error consists of independent increments.
Let $\delta=0.390625$ in (\ref{taudef}) to obtain 
$J=256\,000$ time increments for every $u_0\in{\cal U}$ 
of the form $z_j =\vb Z_j,\phi\ve$
where
$$
	Z_j=R(\tau_j;u_0)-R(\tau_{j-1};u_0)
	\qquad\hbox{and}\qquad\tau_j=j\delta.
$$
Here $\phi$ is a unit vector chosen as in (\ref{zetav}).
The resulting autocorrelations at the point $(\phi,\phi)$ may
be computed using the ratios
$\gamma(\Delta t)/\gamma(0)$ with $\Delta t=\delta\ell$ and
$$
	\gamma(\Delta t)=
		\bigg\langle{1\over J-\ell}\sum_{j=1}^{J-\ell}
			\Big(z_j z_{j+\ell}-\bar z_0\bar z_\ell\Big)\bigg\rangle
\qquad\hbox{where}\qquad
	\bar z_k = \bigg\langle{1\over J-\ell}\sum_{j=1}^{J-\ell} z_{j+k}
\bigg\rangle.
$$
Computational results for the case $\phi=\phi_2$ are depicted in 
Figure \ref{fcor}.  We take $\Delta t\in [0,2\tau]$ where $\tau$ 
is the time for one large-eddy turnover.
Note that $2\tau$ is small compared to total length of each 
time-series, which is greater than $1000\tau$.
When $\Delta t$ is small but positive, the autocorrelation is 
near unity, as expected from (\ref{iscor}).
As $\Delta t$ increases the autocorrelation decreases.
Table \ref{cortable} characterizes the apparent support 
of the autocorrelation by computing the length of the 
smallest interval such that $|\sigma(\Delta t)/\sigma(0)|<0.05$
for all values of $\Delta t$ outside that interval.
When $\alpha_0=0.04$ or $\alpha_0=0.2$ the size of this
interval is about $\tau/4$.
When $\alpha_0=0.01$
the size of the interval is noticeably larger.
This is somewhat
surprising given the slow growth of 
${\cal E}_{\rm rms}^n$ in this case.
Note that using a cutoff of $0.05$ to describe the interval of 
support was somewhat arbitrary---other cutoffs reveal a 
similar relationship between the parameters $\alpha_0$, $d$, $\phi$ 
and the size of the interval.

\begin{figure}[h!]
	\centerline{\begin{minipage}[b]{0.75\textwidth}
    \caption{\label{fcor}%
	Autocorrelation for $\gamma(\Delta t)/\gamma(0)$ 
	when $\phi=\phi_2$.
	The graphs have been offset along the vertical axis
	for clarity.  The vertical line at $\tau$ represents
	the time for one large-eddy turnover.
	}
	\end{minipage}}
	\centerline{\includegraphics[height=0.45\textwidth]{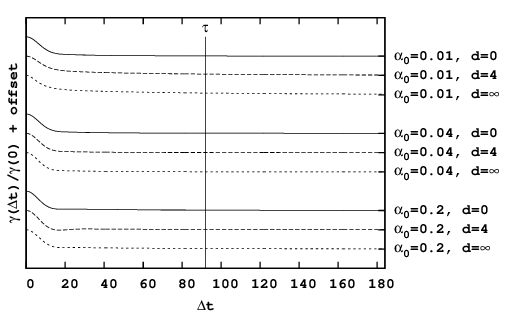}}
\end{figure}

Next, characterize the spatial correlations
in residual error.
Working with the residual vorticity $\rho$
is simpler than $R$ because it is a scalar.
Since the norm on $V$ satisfies
$$
	\big\|R(t)\big\|^2=\int_\Omega \big|\rho(t,x)\big|^2 dx,
$$
then working with $\rho$ is also, in some sense, natural.
For completeness, we also consider spatial correlations in 
residual velocity field $R$ as well.

To begin, let $\delta=6.25$ to obtain $J=16\,000$ 
increments in the residual vorticity
given by $\zeta_j=\rho(\tau_j)-\rho(\tau_{j-1})$ where
$\tau_j=j\delta$.
Figure \ref{spacov} illustrates the correlation 
$\sigma(\Delta x)/\sigma(0)$ where
\begin{equation}\label{spcor}
	\sigma(\Delta x)= \bigg\langle{1\over J}
		\sum_{j=1}^{J}\Big(
		\zeta_j(\pi,\pi) \zeta_j(\pi+\Delta x,\pi)
		-	\overline \zeta(\pi,\pi)\overline \zeta(\pi+\Delta x,\pi)\Big)
	\bigg\rangle
\end{equation}
and
$$
	\overline\zeta(x,y)=
		\bigg\langle
		{1\over J}\sum_{j=1}^{J} \zeta_j(x,y)\bigg\rangle.
$$
These curves may be seen as particular slices of the covariance 
matrix of the residual error in the vorticity at the point 
$(\pi,\pi)$ corresponding to the $x$-axis.
It is interesting that when $\alpha_0=0.01$ and $\alpha_0=0.04$
there is a distance $\Delta x$ for which the correlation is negative.
Nearly identical graphs are obtained when slicing the
covariance matrix at different points in different directions.
In particular, although the presence of a non-zero body force
has the potential to render the statistics of the resulting 
flow neither homogeneous or isotropic,
our results show that the spatial correlations of the residual 
error are nearly homogeneous and isotropic.  
This is consistent with the assumptions of homogeneity and 
isotropy used in the derivations of the turbulence models.

\begin{table}[h!]
	\centerline{\begin{minipage}[b]{0.75\textwidth}
    \caption{\label{cortable}%
	Lengths of the smallest intervals in time outside of
	which the autocorrelation $|\gamma(\Delta t)/\gamma(0)|< 0.05$ 
	for $\phi\in\{\phi_1,\phi_2,\phi_3\}$ and
	the lengths of the smallest intervals in space outside 
	of which the spatial correlations
	$|\sigma(\Delta x)/\sigma(0)|< 0.05$ and
	$|\Tr \sigma_u(\Delta x)/\Tr \sigma_u(0)|< 0.05$.
	}
	\end{minipage}}
	\medskip
    \centering
    \begin{tabular}{ l c | r r r | c c }
    \hline\hline
\vphantom{\raise8pt\hbox{H}}%
%	&&
%		\multicolumn{3}{c|}{Time}& 
%		\multicolumn{2}{c}{Space} \\
    $\phantom{.}\alpha_0$ & $d$ & 
		$\phi_1$ & $\phi_2\phantom{2}$ & $\phi_3\phantom{3}$
		& $\sigma$ & $\Tr \sigma_u$ \\ [0.5ex]
    \hline
\vphantom{\raise8pt\hbox{H}}%
0.01&   0& 17.97& 34.38& 87.50&  0.49&  0.74 \\ 
0.01&   4& 19.92&119.92&150.39&  0.39&  0.39 \\ 
0.01& $\infty$& 22.27&141.41& 86.72&  0.39&  0.39 \\ 
0.04&   0& 19.14& 28.52& 54.69&  0.64&  0.76 \\ 
0.04&   4& 17.19& 22.27& 38.28&  0.74&  0.61 \\ 
0.04& $\infty$& 14.84& 25.78& 39.45&  0.66&  0.47 \\ 
0.2&   0& 23.83& 46.09& 43.75&  1.08&  1.94 \\ 
0.2&   4& 13.67& 12.89& 14.45&  1.01&  2.38 \\ 
0.2& $\infty$& 28.52& 53.52& 40.23&  1.64&  2.38 \\ 
[1ex]
\hline
\end{tabular}
\end{table}

Note that the apparent support of the spatial correlations 
increases
as $\alpha_0$ increases but is less affected by $d$.
Table \ref{cortable} characterizes the support of the
spatial correlation in a way analogous to the method
used for the autocorrelation.
For comparison, the support of the spatial 
correlations in the flow itself, obtained by taking 
$\zeta_j=\omega(\tau_j)-\omega(\tau_{j-1})$ in (\ref{spcor}),
may be characterized by an interval of length $1.74$ outside
of which the correlation is less than~$0.05$.
We conclude the vorticity field of the original flow is spatially 
correlated over distances which are three to six times longer than
the spatial correlations observed in the residual vorticity
when $\alpha_0=0.01$ and $\alpha_0=0.04$.
When $\alpha_0=0.2$ the distances are comparable.

This section finishes by characterizing the spatial correlations of
the residual error in the velocity. 
To avoid the $2$-by-$2$ matrices which arise from
the horizontal and vertical components of $R$,
take the trace to obtain
$$
	\Tr \sigma_u(\Delta x)= \bigg\langle{1\over J}
		\sum_{j=1}^{J}\Big(
		R_j(\pi,\pi) \cdot R_j(\pi+\Delta x,\pi)
		-	\overline R(\pi,\pi)\cdot \overline R(\pi+\Delta x,\pi)\Big)
	\bigg\rangle.
$$
The graphs of $\Tr\sigma_u(\Delta x)/\Tr\sigma_u(0)$ look 
similar to those in Figure \ref{spacov}
except with less compact support.
This fact is quantified by the last column of Table \ref{cortable}.
As with the vorticity, the correlation distances in the 
velocity field of the physical flow 
are three to six times longer than the correlation distances
in the model error when $\alpha_0=0.01$ and $\alpha_0=0.04$.

\begin{figure}[h!]
	\centerline{\begin{minipage}[b]{0.75\textwidth}
    \caption{\label{spacov}%
	Spatial correlation $\sigma(\Delta x)/\sigma(0)$
	for different values of $\alpha_0$ and $d$.  
	The graphs have been offset along the vertical axis
	for clarity.
}
	\end{minipage}}
	\centerline{\includegraphics[height=0.45\textwidth]{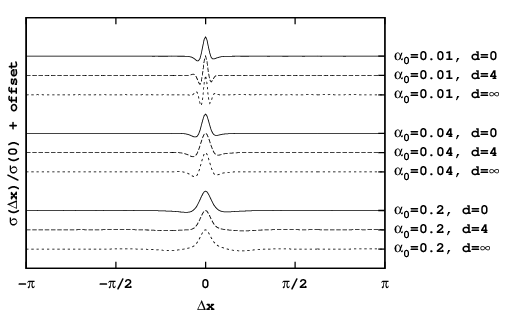}}
\end{figure}

\section{Conclusions and Future Work}

A computational method was developed for studying the model error 
by calculating the time evolution of the root-mean-squared residual 
error taken over an ensemble of trajectories on the global attractor.
We emphasize that the method of computing the model error
employed in this paper
avoids the difficulty that there is no shadowing result which can 
be used to compare separate evolutions of approximate dynamics 
to exact dynamics over long periods of time.
Taking the discrete dynamics given by $S$ to be exact
allows us to focus on the effects of the alpha modeling apart 
from issues relating to the numerical discretization of the 
continuous dynamics.

During our study we identified an
effective averaging length scale $\alpha_0=\alpha/\sqrt{d+1}$ 
in the LANS-alpha and NS-alpha 
deconvolution models of turbulence 
and created 
a new turbulence model, the exponential-alpha model, corresponding to 
the limit as $d\to\infty$.  
This identification 
of $\alpha_0$ allows direct comparison of the residual error
while holding the behavior in the microscales essentially constant.
Numerical computation showed
for a particular time-independent forcing function 
with Grashof number $G=250\,000$ that

\begin{itemize}
\item if $\alpha_0=0.01$ or if $\alpha_0=0.04$ and
$d\in \{4,\infty\}$,
then ${\cal E}_{\rm rms}^n$ grows as $\sqrt{t_n}$;

\item if $\alpha_0=0.20$ or if $\alpha_0=0.04$ and
$d=0$, then ${\cal E}_{\rm rms}^n$ grows linearly as $t_n$.
\end{itemize}
Note that the NS-alpha deconvolution 
model of order $d=4$ performs similarly to the exponential-alpha model
in these experiments.
Little difference is expected, therefore, between NS-alpha 
deconvolution models with $d\ge 4$.

Although the residual errors in the deterministic alpha models 
are differentiable and do not have independent
increments, when $\alpha_0$ is sufficiently small they produce 
time series which are similar to a spatially-correlated and
temporally-white Gaussian process.
At the same time, when $\alpha_0$ is too large, the model error
includes a systematic bias.
Further analysis indicates that this systematic bias is 
concentrated in a direction spanned by the forcing function.
Moreover, when $d=0$ the bias is opposite in sign compared to
when $d=4$.
This suggests for large values of $\alpha_0$ that there may 
exist a fractional value of $d\in (0,4)$ which removes the
bias and for which ${\cal E}_{\rm rms}^n$ grows as $\sqrt{t_n}$.

It would be interesting to further study the analytic properties 
of the new exponential alpha model proposed in this paper, to use the
effective averaging length scale $\alpha_0$ to compare subgrid
scale models and boundary layers for flows with more complicated 
boundary conditions,
to study how the body forcing affects the homogeneity and isotropy
of the statistics of a turbulent flow and to use the 
techniques developed in this paper to create computationally efficient 
turbulence models with model errors that grow only as $\sqrt t$ over time.

\section*{Acknowledgements}

The author was supported in part by NSF grant DMS-1418928.  The author
would like to thank the anonymous referees for many comments 
which improved this paper.  One idea particularly worth mentioning 
was to increase the size of the ensemble averages and perform 
the statistical analysis given in Section \ref{furan}.

\vfill\eject

\end{document}